\documentclass[submission,copyright,creativecommons]{eptcs}
\providecommand{\event}{AiML 2026} 

\usepackage{aiml26}

\usepackage{iftex}
\usepackage{multicol}
\usepackage{amsmath,amssymb,amsthm}
\usepackage{enumitem}
\usepackage{comment}
\usepackage{booktabs}

\ifpdf
  \usepackage{underscore}        
  \usepackage[T1]{fontenc}        
\else
  \usepackage{breakurl}           
\fi

\newcommand{\cl}{\mathrm{cl}}         
\newcommand{\Int}{\mathrm{int}}       
\newcommand{\st}{\mathrm{bcn}}         
\newcommand{\sem}[1]{[\![#1]\!]}      
\newcommand{\diam}{\Diamond}          
\newcommand{\sq}{\Box}                
\newcommand{\Nbd}{\mathcal{N}}         
\newcommand{\str}{{{}^{*}\!}}        
\newcommand{\lozengens}{\diam_\mathrm{ns}}
\newcommand{\lozengest}{\lozenge_\mathrm{st}}
\usepackage{xcolor}
\usepackage{apxproof}

\title{Halo Semantics for Modal Logic}
\author{Yo\`av Montacute
\institute{National Institute of Informatics \\ Tokyo, Japan}
\email{montacute@nii.ac.jp}}

\newcommand{\titlerunning}{Halo Semantics for Modal Logic}
\newcommand{\authorrunning}{Yo\`av Montacute}

\hypersetup{
  bookmarksnumbered,
  pdftitle    = {\titlerunning},
  pdfauthor   = {\authorrunning},
  pdfsubject  = {EPTCS},               
  pdfkeywords = {Nonstandard analysis, Halo, Topological semantics, Modal logic} 
}

\begin{document}

\maketitle

\begin{abstract}
    In nonstandard analysis the halo of a point in a topological space 
is the intersection of the nonstandard extensions of all its open neighbourhoods.
We define a parametric family of modal operators from the halo by varying which elements of the nonstandard extension are admitted as witnesses, and identify four canonical instances.
Two recover well-known modalities: the topological closure and the Cantor derivative.
A third reduces to Kripke semantics over the specialisation preorder.

The fourth, purely nonstandard instance admits only nonstandard witnesses.
The Transfer Principle forces it to coincide with the $\omega$-accumulation point operator, a classical topological notion not previously studied in modal logic.
Unlike the Cantor derivative, the $\omega$-accumulation operator maps arbitrary sets to closed sets without any separation axiom, yielding an $\omega$-Cantor--Bendixson decomposition on all topological spaces.
Axiom~4 holds universally, again without separation conditions.
We prove that $\mathsf{K4}$ is the complete logic over infinite spaces, and $\mathsf{GL}$ over infinite $\omega$-scattered spaces.
\end{abstract}

\section{Introduction}

In Robinson's nonstandard analysis, every point $x$ in a topological space acquires an infinitesimal neighbourhood: the intersection of the nonstandard extensions of all its open sets, called the \emph{halo}.
It replaces the infinitary quantification over all open neighbourhoods that defines topological operations with a single membership test against one set.
The halo encodes the local topology of the space in a single object.
This invites the following inquiry: 
If a formula $\varphi$ is true at various points of the space, then what is the relationship between the halo of a given point and the nonstandard extension $\str\sem{\varphi}$ of the truth set of such a formula?

We investigate this by defining a family of modal operators which specify for each point $x$, a set $E(x)$ of witnesses which are excluded from the nonstandard truth set $\str\sem{\varphi}$.
The resulting family has four canonical members, on which we focus.
The purely standard case collapses to Kripke semantics over the specialisation preorder.
Two more turn out to coincide with well-known modalities: when nothing is excluded the operator equates to topological closure, and when the point itself is excluded it equates to the Cantor derivative.
Hence, the classical semantics of McKinsey-Tarski~\cite{McKinseyTarski1944} and 
Esakia~\cite{Esakia1981} arise naturally from our framework. 

The truly novel operator is the one obtained by excluding all standard points, so that witnesses must be nonstandard hyperpoints.
The \emph{Transfer Principle} of nonstandard analysis forces this condition to be equivalent to a well-known but modally unstudied topological notion: $\omega$-accumulation.
Unlike the Cantor derivative, the $\omega$-accumulation operator maps arbitrary sets to closed sets without any separation axiom.
This yields a Cantor--Bendixson-style decomposition on all topological spaces, paralleling the classical theorem.
In addition, Axiom~4 holds on all topological spaces, again with no separation condition required.
We show that the logic $\mathsf{K4}$ is complete for the nonstandard semantics over all infinite topological spaces; and that $\mathsf{GL}$ is complete over all infinite $\omega$-scattered spaces.

\paragraph{Outline.}
Section~\ref{sec:ns} recalls the relevant definitions from nonstandard analysis.
Section~\ref{sec:param} defines the parametric family and analyses all four exclusion choices with examples.
Section~\ref{subsec:st} studies the purely standard operator.
Section~\ref{sec:closure} discusses the closure and Cantor derivative operators.
Section~\ref{sec:ns-op} develops the theory of the nonstandard modality, including the $\omega$-accumulation characterisation, its coincidence with the Cantor derivative on $T_1$ spaces, the $\omega$-Cantor-Bendixson decomposition, and the completeness of $\mathsf{K4}$ and $\mathsf{GL}$.
We conclude with some remarks.

\section{Nonstandard Framework}
\label{sec:ns}

We recall the axiomatic framework for nonstandard analysis following Salbany and Todorov \cite{SalbanyTodorov1999}, who work within the superstructure of an infinite set. 
This framework specifies the conditions required for halos to have the appropriate topological properties.

\subsection{Superstructures and the star map}

Let $S$ be an infinite set. In our applications $S = X \cup \mathbb{R}$, where $(X,\tau)$ is a topological space.
The \emph{superstructure} over $S$ is the hierarchy
\[
  V(S) = \bigcup_{k \in \mathbb{N}_0} V_k(S),
\]
built inductively as
\[
  V_0(S) = S, \qquad V_{k+1}(S) = V_k(S) \cup \mathcal{P}(V_k(S)).
\]
Each level adds all subsets of the previous level. The idea is that $V_0(S)$ contains the base objects, $V_1(S)$ adds all sets of base objects, $V_2(S)$ adds all sets of those, and so on. Concretely:

\begin{itemize}
  \item $V_0(S) = S$: the points of $X$ and the real numbers.
  \item $V_1(S) = S \cup \mathcal{P}(S)$: includes all subsets of $S$,
    in particular every open set $U \subseteq X$, since $U \subseteq X
    \subseteq S$ and hence $U \in \mathcal{P}(S) \subseteq V_1(S)$.
  \item $V_2(S)$: includes all subsets of $V_1(S)$, in particular the topology $\tau$, since $\tau$ is a collection of open sets and each open set is in $V_1(S)$, so $\tau \subseteq V_1(S)$ and hence $\tau \in \mathcal{P}(V_1(S)) \subseteq V_2(S)$.
\end{itemize}

Functions are identified with their graphs (sets of ordered pairs), and ordered pairs are encoded as sets in the usual way ($\langle a,b\rangle = \{\{a\},\{a,b\}\}$), so functions also appear at a finite level of $V(S)$.
Every mathematical object used in this paper, such as points, open sets, the topology, valuations, and the neighbourhood filter, is an element of $V(S)$ at some finite level.

The formal language $\mathcal{L}(V(S))$ consists of \emph{bounded quantifier formulas}: formulas built from $=$, $\in$, the logical connectives, and quantifiers restricted to the form $\forall x \in A$ or $\exists x \in A$ for some $A \in V(S)$.
Restricting to bounded quantifiers ensures that the formula does not quantify over the entire (proper-class-sized) universe, only over specific sets already in $V(S)$. 
This language is expressive enough to state all standard topological properties of $(X,\tau)$
(see \cite{SalbanyTodorov1999}).

\begin{definition}[nonstandard model]
\label{def:nsmodel}
A \emph{nonstandard model} of $V(S)$ is a superstructure $V(\str S)$ together with a \emph{star map} $\str : V(S) \to V(\str S)$ satisfying the axioms $\mathrm{NSA1}$, $\mathrm{NSA2}$ and $\mathrm{NSA3}$ below.

A set $B \in V(\str S)$ is called \emph{internal} if $B \in \str C$ for some $C \in V(S)$, i.e.\ if $B$ is an element of the nonstandard extension of some standard set. A set that is not internal is called \emph{external}.

\begin{enumerate}[label=\textup{(NSA\arabic*)}]
  \item \textbf{Extension Principle.}
    $\str s = s$ for all $s \in S$; equivalently, $S \subseteq \str S$.

  \item \textbf{Transfer Principle.}
    A bounded quantifier formula $\Phi(A_1,\ldots,A_n)$ is true in $V(S)$ if and only if $\Phi(\str A_1,\ldots,\str A_n)$ is true in $V(\str S)$.

  \item \textbf{Saturation Principle.}
    Every family of internal sets of cardinality at most $\kappa\geq |V(S)|$ with the finite intersection property has non-empty intersection.
\end{enumerate}
We call a model satisfying the axioms above 
\emph{polysaturated}.
Such models exist for any infinite $S$, as shown by the ``ultrapower construction''~\cite{SalbanyTodorov1999}. 
We fix one from this point onward.
\end{definition}

The distinction between internal and external sets is fundamental to nonstandard analysis.
Any set explicitly constructed by Transfer from a standard formula, e.g.\ $\str(U \cap A)$, $\str U \setminus \str A$, and so on is internal.
The Transfer Principle applies only to internal sets as bounded quantifier formulas can be transferred between $V(S)$ and $V(\str S)$ only when all sets quantified over are internal. 
Similarly, the Saturation Principle produces intersections only for families of internal sets. 
Important examples of external sets used in this paper are the standard universe $X$ inside $\str X$, the halo $\mu(x)$ and the set $\str S\setminus X$ for infinite $S$.

For us this distinction is important when applying Saturation to a family such as $\{\str (U \cap A) \setminus G : U \in \mathcal{N}(x),\, \text{finite } G \subseteq X\}$, where $\Nbd(x)=\{U \in \tau:x\in U\}$ is the open neighbourhood filter of $x$, because each individual member $\str (U \cap A) \setminus G$ is internal;
the intersection of this family, however, need not be internal, and indeed, it is a subset of $\str A \setminus X$, which is external.

The Transfer Principle immediately yields the following structural properties of $\str$, which we use throughout without further comment.

\begin{lemma}[properties of the star map]
\label{lem:star}
For all $A, B \subseteq X$ and $x \in X$:
\begin{enumerate}[label=\textup{(\roman*)}]
  \item $\str(A \cap B) = \str A \cap \str B$;
  \item $\str(A \cup B) = \str A \cup \str B$;
  \item $\str(X \setminus A) = \str X \setminus \str A$;
  \item $A \subseteq B \Rightarrow \str A \subseteq \str B$\quad
    \textup{(monotonicity)};
  \item $\str\emptyset = \emptyset$ and $\str X \supseteq  X$;
  \item $x \in A$ if and only if $x \in \str A$, for $x \in X$.
\end{enumerate}
\end{lemma}

\begin{proof}
Each statement is a direct application of the Transfer Principle to the corresponding bounded quantifier formula expressing set operations and membership. 
For instance, (i) follows by transferring 
$\forall z \in X\,(z \in A \cap B \leftrightarrow z \in A \wedge z \in B)$, 
which yields 
$\forall z \in \str X\,(z \in \str(A \cap B) \leftrightarrow z \in \str A \wedge z \in \str B)$.
In part (v), $\str\emptyset = \emptyset$ follows by Transfer and $\str X \supseteq X$ from Extension.
Part (vi) holds because $x \in S \subseteq \str S$ by Extension and the formula $x \in A$ is quantifier-free. 
It readily follows by transfer that $x \in A$ if and only if $x \in \str A$.
\end{proof}

\subsection{The halo and nonstandard characterisations}

Fix a topological space $(X, \tau)$. 

\begin{definition}[halo]
\label{def:halo}
For $x \in X$, the \emph{halo} of $x$ is defined as
\[
  \mu(x) = \bigcap_{U \in \Nbd(x)} \str U.
\]
\end{definition}
This set is also called the \emph{monad} of $x$ in the 
literature~\cite{Robinson1966,SalbanyTodorov1999}.
\begin{lemma}
\label{lem:halorefl}
For all $x \in X$, $x \in \mu(x)$.
\end{lemma}

\begin{proof}
For any $U \in \Nbd(x)$ we have $x \in U$, hence $x \in \str U$ by Lemma~\ref{lem:star}(vi). 
Since this holds for every $U \in \Nbd(x)$, $x \in \bigcap_{U \in \Nbd(x)} \str U = \mu(x)$.
\end{proof}
The following is immediate from the definition of a halo.
\begin{lemma}
\label{lem:halomono}
If $\tau' \subseteq \tau$, then $\mu_\tau(x) \subseteq \mu_{\tau'}(x)$.
\end{lemma}

\begin{definition}[beacon]
\label{def:st}
For $B \subseteq \str X$, the \emph{beacon} of $B$ is defined as
\[
  \st(B) = \{ x \in X \mid \mu(x) \cap B \neq \emptyset \}.
\]
\end{definition}

Each standard point $x$ casts a halo $\mu(x)$ into $\str X$. 
The beacon of $B$ is the set of standard points whose halos reach $B$. 
As we show below, the beacon of $\str A$ is precisely the topological closure of $A$.


The following proposition collects the fundamental nonstandard equivalences.
Parts (i) and (ii) are due to Robinson \cite{Robinson1966}
(see also \cite{Luxemburg1969}); the remaining parts are standard consequences recorded, e.g.\ in
\cite{SalbanyTodorov1999}.

\begin{proposition}[nonstandard characterisations]
\label{prop:nschar}
Let $(X,\tau)$ be a topological space with polysaturated nonstandard extension. For $A \subseteq X$ and $x \in X$:
\begin{enumerate}[label=\textup{(\roman*)}]
  \item Closure:
    $x \in \cl(A)$ iff $\mu(x) \cap \str A \neq \emptyset$.
    Equivalently, $\cl(A) = \st(\str A)$. 
    So $A$ is closed  iff $A=\st(\str A)$.

  \item Open set:
    $A \in \tau$ iff $\mu(x) \subseteq \str A$ for every $x \in A$.

  \item Limit point:
    $x$ is a limit point of $A$ iff $\mu(x) \cap (\str A \setminus \{x\}) \neq \emptyset$.

  \item Isolated point:
    $x \in A$ is isolated in $A$ iff $\mu(x)\cap\str A= \{x\}$.

  \item $T_1$ space:
    $(X,\tau)$ is $T_1$ iff $\mu(x) \cap X = \{x\}$ for every $x \in X$.

  \item Hausdorff space:
    $(X,\tau)$ is Hausdorff iff $\mu(x) \cap \mu(y) = \emptyset$ for all distinct $x, y \in X$.
\end{enumerate}
\end{proposition}

\section{The Halo Operator}
\label{sec:param}

\subsection{Logic}

Let $\mathbf{Prop} = \{p, q, r, \ldots\}$ denote the set of
propositional variables. 
The modal language has a single diamond primitive
$\diam$.
The language is defined inductively as follows:
\[
  \varphi \;::=\; p \;\mid\; \bot \;\mid\; \neg\varphi \;\mid\;
    \varphi \wedge \psi \;\mid\; \diam\varphi.
\]
Boolean connectives $\top, \vee, \to, \leftrightarrow$ are standard abbreviations. The box is always the De Morgan dual:
$\sq\varphi := \neg\diam\neg\varphi$.
The axiomatic systems used in this paper are all normal modal logics, i.e.\ include the Axiom~K. 
We recall the relevant axiom schemas:
    \begin{itemize}
  \item[(K)] $\sq(\varphi \to \psi) \to (\sq\varphi \to \sq\psi)$
  \item[(T)] $\sq\varphi \to \varphi$ 
  \item[(4)] $\sq\varphi \to \sq\sq\varphi$ 
  \item[(w4)] $\diam\diam\varphi \to \diam\varphi \vee \varphi$
  \item[(L)] $\square(\square\varphi\to\varphi)\to \square \varphi$
\end{itemize}
All of the axiomatic systems discussed in this paper include classical tautologies, Modus Ponens and Necessitation.  
The axiomatic systems we consider are:
\begin{center}
    $\mathsf{S4}$ = K + T + 4; $\;\mathsf{K4}$ = K + 4; $\;\mathsf{wK4}$ = K + w4; $\;\mathsf{GL}$ = K + L.
\end{center} 
\begin{definition}[topological model]
\label{def:model}
A \emph{topological model} is a pair $\mathcal{M} = ((X,\tau), \nu)$ where $(X,\tau)$ is a topological space and $\nu : \mathbf{Prop} \to \mathcal{P}(X)$ is a valuation. 
We work inside a fixed polysaturated nonstandard extension
of $X$.

The \emph{extension} $\sem{\varphi}^{\mathcal{M}} \subseteq X$ is defined inductively in the obvious way on Boolean connectives, with the modal clause determined by the choice of \emph{exclusion function}.
\end{definition}

\begin{definition}[halo operator]
\label{def:excl}
Given a topological model $\mathcal{M}$ and an exclusion function $E:X\to\mathcal{P}(\str X)$, the \emph{halo
operator} $\diam_E$ is defined by the semantic clause:
\[
  \sem{\diam_E\varphi}^{\mathcal{M}}
  = \bigl\{x \in X \;\bigm|\;
    \mu(x) \cap \bigl(\str \sem{\varphi}^{\mathcal{M}} \setminus E(x)\bigr)
    \neq \emptyset \bigr\}.
\]
We write $x \models \varphi$ for $x \in \sem{\varphi}^{\mathcal{M}}$ under the semantics determined by $E$.
We omit the superscript $\mathcal{M}$ when clear from context.
\end{definition}

We focus on four choices for $E$ in this paper: the purely standard case, where $E(x)=\str X\setminus X$; the closure case, where $E(x)=\varnothing$; the Cantor derivative case, where $E(x)=\{x\}$; and the purely nonstandard case, where $E(x)=X$.
We denote the resulting modalities by $\lozenge_{\rm st}$, $\lozenge_c$, $\lozenge_d$ and $\lozengens$, respectively.
The names will be justified in Section~\ref{sec:closure}, where we show that the two operators $\lozenge_c$ and $\lozenge_d$ coincide with the closure and the Cantor derivative, respectively.

We establish the relationship between the four modalities.
\begin{proposition}
\label{prop:order}
For every topological model $\mathcal{M}$ and formula $\varphi$:
\[
  \sem{\lozengens\varphi}
  \;\subseteq\;
  \sem{\diam_d\varphi}
  \;\subseteq\;
  \sem{\diam_c\varphi}.
\]
The purely standard operator satisfies $\sem{\diam_\mathrm{st}\varphi} \subseteq
\sem{\diam_c\varphi}$, but is incomparable with $\diam_d$.
\end{proposition}

\begin{proof}
The inclusions $\sem{\lozengens\varphi} \subseteq \sem{\diam_d\varphi} \subseteq \sem{\diam_c\varphi}$ follow from
$\str A \setminus X \subseteq \str A
\setminus \{x\} \subseteq \str A$.
For $\sem{\diam_\mathrm{st}\varphi} \subseteq \sem{\diam_c\varphi}$, every standard witness in
$\mu(x) \cap \sem{\varphi}$ is also a witness for $\diam_c$, so $\sem{\diam_\mathrm{st}\varphi} \subseteq \sem{\diam_c\varphi}$.

For incomparability, consider an arbitrary topological space $(X,\tau)$, a point $x \in X$, and set $\nu(p) = \{x\}$. 
Since $x \in \mu(x)$ always, we have $\mu(x) \cap \sem{p} = \{x\} \neq \emptyset$, so
$x \models \diam_\mathrm{st} p$. But since
$\mu(x)\cap (\{x\} \setminus \{x\}) = \emptyset$ then
$x \not\models \diam_d p$.
Hence $\sem{\lozengest\varphi} \not\subseteq \sem{\lozenge_d\varphi}$ in general.

For the other direction, consider $\mathbb{R}$ with the usual topology, let $x = 0$ and $\nu(p) = \{1/n : n \geq 1\}$. 
Then since every neighbourhood of $0$ contains points of $\nu(p) \setminus \{0\}$, $0 \models \diam_d p$. 
But $\mathbb{R}$ is $T_1$, so
$\mu(0) \cap X = \{0\}$ by Proposition~\ref{prop:nschar}(v).
Since $0 \notin \nu(p)$, it follows that $\mu(0) \cap \sem{p} = \emptyset$, and so
$0 \not\models \diam_\mathrm{st} p$.
Hence $\sem{\lozenge_d\varphi} \not\subseteq \sem{\lozengest\varphi}$ in general.
\end{proof}
\begin{proposition}
\label{prop:decomp}
For every topological model $\mathcal{M}$ and formula $\varphi$:
\[
  \sem{\diam_c\varphi} = \sem{\diam_\mathrm{st}\varphi} \cup \sem{\lozengens\varphi}.
\]
\end{proposition}

\begin{proof}
The inclusion $\supseteq$ follows from Proposition~\ref{prop:order}.
For $\subseteq$, suppose $x \in \sem{\diam_c\varphi}$, so there exists $y \in \mu(x) \cap \str\sem{\varphi}$.
Either $y \in X$ or $y \notin X$.
In the first case, $y \in \mu(x) \cap \sem{\varphi}$ by Lemma~\ref{lem:star}(vi), so $x \in \sem{\diam_\mathrm{st}\varphi}$.
In the second case $y \in \mu(x) \cap (\str\sem{\varphi} \setminus X)$, so $x \in \sem{\lozengens\varphi}$.
\end{proof}
\subsection{Examples}
\label{sec:examples}

We illustrate all four operators on two examples. The first is a $T_1$ space
($\mathbb{R}$), which demonstrates the collapse $\lozengens = \diam_d$
and the incomparability of $\diam_\mathrm{st}$ with $\diam_d$. The second is
a non-$T_1$ space (the Sierpi\'{n}ski space), chosen specifically to separate $\diam_d$ from $\lozengens$.

\begin{example}[the real line]
\label{ex:R}
Let $X = \mathbb{R}$ with the standard topology and $\nu(p) = (0,1) \cup \{2\}$.
Since $\mathbb{R}$ is $T_1$, Proposition~\ref{prop:nschar}(v) implies $\mu(x) \cap \mathbb{R} = \{x\}$ for every $x$. 
The halo  $\mu(x) = \{y \in \str \mathbb{R} : |y-x| < r \text{ for all standard } r>0\}$ consists of all hyperreals infinitesimally close to $x$.
By Transfer, $\str(0,1) = \{y \in \str \mathbb{R} : 0 < y < 1\}$ and $\str\{c\} = \{c\}$, for each $c \in \mathbb{R}$.

At $x=0$, each positive infinitesimal $\varepsilon$ satisfies $\varepsilon \in \mu(0)$,
$\varepsilon \in \str(0,1) \subset \str \nu(p)$, and $\varepsilon \notin \mathbb{R}$.
Hence $0 \in \sem{\diam_c p}$; since $\varepsilon \neq 0$, $0 \in \sem{\diam_d p}$; since $\varepsilon \notin \mathbb{R}$, $0 \in \sem{\lozengens p}$. 
But $0 \notin \nu(p)$, so $0 \notin \sem{\diam_\mathrm{st} p}$.

At $x=2$, the point $2$ is isolated in $\nu(p)$.
For each $y \in \mu(2)$ with $y \neq 2$, $|y-2|$ is a positive infinitesimal, so $y \notin \str(0,1)$.
Thus $\mu(2) \cap (\str \nu(p) \setminus \{2\}) = \emptyset$, 
so $2 \notin \sem{\diam_d p}$ and $2 \notin \sem{\lozengens p}$. 
However $2 \in \nu(p)$, so $2 \in \sem{\diam_\mathrm{st} p}$; and $2 \in \cl(\nu(p))$, so $2 \in \sem{\diam_c p}$.
\end{example}
The example above does not separate $\lozenge_d$ from $\lozenge_{ns}$.
We show that the two can be separated even on a two-point topological space.
\begin{example}[the Sierpi\'{n}ski space]
\label{ex:sierpinski}
Let $X = \{a,b\}$ with topology $\tau = \{\emptyset, \{a\}, \{a,b\}\}$.
The open sets containing $a$ are $\{a\}$ and $X$; the only open set containing $b$ is $X$. This space is not $T_1$: the singleton $\{a\}$ is open, hence $\{b\} = X \setminus \{a\}$ is closed, but
$\{a\} = X \setminus \{b\}$ is not closed since $\{b\}$ is not open.
Let $\nu(p) = \{a\}$.
Since $\{a\}$ and $X$ are the open sets containing $a$, we obtain
 $ \mu(a) = \str\{a\} \cap \str X = \str\{a\}$.
Since $X$ is the only open set containing $b$, $
  \mu(b) = \str X$.
Recall that by Transfer finite sets are fixed by $\str$.
In particular, $\str \nu(p) = \str\{a\} = \{a\}$.
So $a$ is in $\sem{ \lozengest p}$ and $\sem{\lozenge_c p}$ but not in $\sem{\lozenge_d p}$ or $\sem{\lozengens p}$.

At $x = b$, we have $\mu(b) \cap \{a\} = \{a\} \neq \emptyset$ and 
$\mu(b) \cap \str\{a\} = \{a\} \neq \emptyset$, so 
$b \models \diam_\mathrm{st} p$ and $b \models \diam_c p$.
Also $\mu(b) \cap (\{a\} \setminus \{b\}) = \{a\} \neq \emptyset$, 
so $b \models \diam_d p$.
But $\mu(b) \cap (\{a\} \setminus X) = \{a\} \setminus \{a,b\} = \emptyset$, 
so $b \not\models \lozengens p$.
\end{example}

\section{The purely standard operator}
\label{subsec:st}

The purely standard modality $\diam_\mathrm{st}$ is obtained by setting 
$E(x) = \str X \setminus X$, so that only standard points are admitted as witnesses.
\begin{proposition}
\label{prop:stgeneral}
On arbitrary topological spaces $(X,\tau)$, the operator $\diam_\mathrm{st}$ computes exactly the Kripke diamond over the \emph{specialisation preorder} $\leq_\tau$ defined by $x \leq_\tau y$ iff $x \in \cl(\{y\})$, i.e.\
\[
  x \models \diam_\mathrm{st}\varphi
  \iff
  \exists\, y \in X\;:\; x \leq_\tau y \text{ and } y \models \varphi.
\]
\end{proposition}

\begin{proof}
Using $\str \sem{\varphi} \subseteq \str X$ and $\str A \cap X = A$ for all $A \subseteq X$ (Lemma~\ref{lem:star}(vi)):
\[
  x \models \diam_\mathrm{st}\varphi
  \iff
  \mu(x) \cap \bigl(\str \sem{\varphi} \setminus (\str X \setminus X)\bigr) \neq \emptyset
  \iff
  \mu(x) \cap \sem{\varphi} \neq \emptyset.
\]
For $y \in X$, $y \in \mu(x)$ iff $y \in \str U$ for every $U \in \Nbd(x)$, and by Lemma~\ref{lem:star}(vi), $y \in \str U$ iff $y \in U$ for standard $y$ and $U$.
Hence $\mu(x) \cap X = {\uparrow}x := \{y \in X : x \leq_\tau y\}$, and so
\[
  x \models \diam_\mathrm{st}\varphi
  \iff \mu(x) \cap \sem{\varphi} \neq \emptyset
  \iff {\uparrow}x \cap \sem{\varphi} \neq \emptyset
  \iff \exists\, y \geq_\tau x\;(y \models \varphi). \qedhere
\]
\end{proof}

\begin{corollary}
\label{cor:stdeg}
On every $T_1$ topological space, $\diam_\mathrm{st}\varphi \equiv \varphi$ for every formula $\varphi$.
\end{corollary}

\begin{proof}
On a $T_1$ space, $\cl(\{x\}) = \{x\}$ for every $x$, so $\leq_\tau$ is equality.
Hence $x\models \diam_\mathrm{st}\varphi$
 iff $x \models \varphi$.
\end{proof}
Recall that a topological space is \emph{Alexandroff} if arbitrary intersections of open sets are open; equivalently, every point has a smallest open neighbourhood.
Every finite space is Alexandroff.
\begin{corollary}
\label{cor:stalexandroff}
On an Alexandroff topological space $(X,\tau)$, the operator $\diam_\mathrm{st}$ coincides with the standard Kripke semantics for the specialisation preorder $\leq_\tau$.
The complete modal logic of $\diam_\mathrm{st}$ over  Alexandroff spaces is $\mathsf{S4}$.
\end{corollary}

\begin{proof}
In an Alexandroff space every point has a smallest open neighbourhood $\mathcal{U}(x)$, so $\mu(x) \cap X = \mathcal{U}(x)$, which is precisely the upset $\{y : x \leq_\tau y\}$.
By Proposition~\ref{prop:stgeneral}, $\diam_\mathrm{st}$ is the Kripke diamond over the preorder $(X, \leq_\tau)$, which is a reflexive, transitive frame.
The complete logic of such frames is $\mathsf{S4}$ \cite{BlackburnDeRijkeVenema2001}.
\end{proof}

\section{The closure operator and the Cantor derivative operator}
\label{sec:closure}

It readily follows that the operator $\diam_c$ is exactly
the topological closure.

\begin{proposition}
\label{prop:clequiv}
For every model $\mathcal{M}$ and formula $\varphi$:
\[
  \sem{\diam_c\varphi}^{\mathcal{M}} = \cl(\sem{\varphi}^{\mathcal{M}}).
\]
In particular, $\sem{\sq_c\varphi}^{\mathcal{M}} = \Int(\sem{\varphi}^{\mathcal{M}})$.
\end{proposition}
\begin{proof}
This is immediate from Proposition~\ref{prop:nschar}(i) with $A = \sem{\varphi}$.
\end{proof}
Soundness and completeness for the closure semantics 
are due to McKinsey and Tarski~\cite{McKinseyTarski1944}.
\begin{theorem}[\cite{McKinseyTarski1944}]
$\mathsf{S4}$ is sound and complete for all (finite) spaces with respect to the closure semantics.
\end{theorem}
Similarly, the operator $\diam_d$ is exactly the Cantor derivative operator.

\begin{proposition}
\label{prop:dequiv}
For every model $\mathcal{M}$ and formula $\varphi$:
\[
  \sem{\diam_d\varphi}^{\mathcal{M}} = d(\sem{\varphi}^{\mathcal{M}}),
\]
where $d(A) = \{x \in X : x \text{ is a limit point of } A\}$
is the Cantor derivative operator.
\end{proposition}

\begin{proof}
This is immediate from Proposition~\ref{prop:nschar}(iii) with $A = \sem{\varphi}$.
\end{proof}

Soundness and completeness for the Cantor derivative semantics originate with Esakia~\cite{Esakia1981}. 

\begin{theorem}[{\cite{Esakia1981}}]
\label{thm:soundd}
\mbox{}
\begin{enumerate}[label=\textup{(\roman*)}]
  \item $\mathsf{wK4}$ is sound and complete for all (finite)  spaces with respect to the Cantor derivative semantics. 
  \item $\mathsf{K4}$ is sound and complete for all $T_D$
    spaces with respect to the Cantor derivative semantics.
\end{enumerate}
\end{theorem}

The following halo-based proof of Axiom~$\mathsf{w4}$ on all topological spaces demonstrates the internal mechanisms of the halo semantics.

\begin{proposition}
\label{thm:wK4}
For every topological space and each model $\mathcal{M}$:
\[
  \sem{\diam_d\diam_d\varphi} \subseteq \sem{\diam_d\varphi} \cup \sem{\varphi}.
\]
\end{proposition}
\begin{proof}
Let $A = \sem{\varphi}$ and suppose $x \in d(d(A))$.
By Proposition~\ref{prop:nschar}(iii), there exists $y \in \mu(x) \cap (\str d(A) \setminus \{x\})$.
Since $y \in \str d(A)$, from Transfer of the bounded formula defining $d(A)$, we obtain that for every $V \in \str\tau$, if $y \in V$ then $V \cap (\str A \setminus \{y\}) \neq \emptyset$.
Since $y \in \mu(x)$, we have $y \in \str U$ for every $U \in \Nbd(x)$.
Each $\str U$ belongs to $\str\tau$ by Transfer, so the transferred condition yields a point $z_U \in \str U \cap (\str A \setminus \{y\})$ for each $U \in \Nbd(x)$.

If $x \in \str A$ then $x \in A$ by Lemma~\ref{lem:star}(vi), so $x \in \sem{\varphi}$ and we are done.
Otherwise $x \notin \str A$.
Then $z_U \neq x$ for every $U$, since $z_U \in \str A$.
Since also $z_U \neq y$, we obtain $z_U \in \str(U \cap A) \setminus \{x\}$ for each $U \in \Nbd(x)$.
Note that each set $\str(U \cap A) \setminus \{x\}$ is internal. 
For any $U_1, \ldots, U_n \in \Nbd(x)$, setting $U = U_1 \cap \cdots \cap U_n \in \Nbd(x)$ implies $z_U \in \str(U \cap A) \setminus \{x\} \subseteq \str(U_i \cap A) \setminus \{x\}$ for each $i$, so the family has the finite intersection property.
By Saturation, there exists $w \in \bigcap_{U \in \Nbd(x)} (\str(U \cap A) \setminus \{x\})$.
Then $w \in \str U$ for every $U \in \Nbd(x)$, so $w \in \mu(x)$; and $w \in \str A$ with $w \neq x$, so $w \in \mu(x) \cap (\str A \setminus \{x\})$ and $x \in d(A)$.
\end{proof}

For $\lozengens$, the stronger inclusion 
$\omega(\omega(A)) \subseteq \omega(A)$ holds outright 
(Lemma~\ref{thm:nsK4}), giving the full Axiom~4 without 
the disjunct $\sem{\varphi}$.
Strengthening w4 to the full Axiom~4 for $\diam_d$ requires 
$T_D$~\cite{Aull1962,BEG2005}.
In the next section, we show that this is not the case for 
the nonstandard semantics.

\section{The nonstandard operator}
\label{sec:ns-op}


Recall the definition of the nonstandard operator:
\[
  x \models \lozengens\varphi
  \iff
  \mu(x) \cap (\str\sem{\varphi} \setminus X) \neq \emptyset.
\]
We show some basic validities of $\lozengens$. 
\begin{proposition}
\label{prop:nsbasic}
For every topological model $\mathcal{M}$ the following properties hold:
\begin{enumerate}[label=\textup{(\roman*)}]
  \item $\models \neg\lozengens\bot$;
  \item $\models \lozengens(\varphi \vee \psi)
    \leftrightarrow (\lozengens\varphi \vee \lozengens\psi)$;
  \item if $\models \varphi \to \psi$ then
    $\models \lozengens\varphi \to \lozengens\psi$;
  \item $\models \lozengens\varphi \to \diam_c\varphi$;
  \item $\models \lozengens\varphi \to \diam_d\varphi$.
\end{enumerate}
\end{proposition}

\begin{proof}
(i) follows from $\str \emptyset \setminus X = \emptyset$;
(ii) follows from $(\str(\sem{\varphi} \cup \sem{\psi})) \setminus X
= (\str \sem{\varphi} \cup \str \sem{\psi}) \setminus X
= (\str \sem{\varphi} \setminus X) \cup (\str \sem{\psi} \setminus X)$;
(iii) follows since $\sem{\varphi} \subseteq \sem{\psi}$ entails $\str \sem{\varphi}
\subseteq \str \sem{\psi}$ and so $\str \sem{\varphi} \setminus X
\subseteq \str \sem{\psi} \setminus X$;
(iv) follows from $\str \sem{\varphi} \setminus X \subseteq \str \sem{\varphi}$; and
(v) follows from $\str \sem{\varphi} \setminus X \subseteq \str \sem{\varphi} \setminus \{x\}$.
\end{proof}

\begin{proposition}
\label{prop:nsrefl}
Axiom~T is not valid for $\lozengens$, even on $\mathbb{R}$.
\end{proposition}

\begin{proof}
$\nu(p) = \{0\}$. 
By Transfer $\str\{0\} = \{0\}$, so $\str\{0\} \setminus \mathbb{R} = \emptyset$.
It follows that $0 \models p$ but $0 \not\models \lozengens p$.
\end{proof}

\subsection{The \texorpdfstring{$\omega$}{omega}-Accumulation Characterisation}

We show that $\lozengens$ has a classical topological characterisation analogous to those of $\diam_c$ and $\diam_d$.
Recall that a point $x \in X$ is an \emph{$\omega$-accumulation point} of $S \subseteq X$ if $|U \cap S| \geq \aleph_0$ for every open set $U$ containing $x$.
We denote the set of $\omega$-accumulation points of $S$ by $\omega(S)$.
This notion appears classically in connection with perfect sets and the Cantor--Bendixson theorem~\cite{Kechris1995}.

Before providing the full characterisation, we need to show the following.
\begin{proposition}\label{keyprop}
    Given $S \subseteq X$, the set $\str S \setminus X$ is non-empty if, and only if, $S$ is infinite.
\end{proposition}
\begin{proof}
 $(\Rightarrow)$ If $S$ is finite, say $|S|=n$, then there exists a bijection $f \colon \{1,\ldots,n\} \xrightarrow{\sim} S$.
By Transfer, there exists a bijection $f \colon \{1,\ldots,n\} \xrightarrow{\sim} \str S$, so $|\str S|=n$. 
Hence $\str S$ consists exactly of the standard images of the elements of $S$ by the Extension Principle, and therefore $\str S = S \subseteq X$, and so $\str S \setminus X = \emptyset$.

$(\Leftarrow)$ Suppose $S$ is infinite. Consider the family
\[
  \mathcal{G} = \bigl\{\str(S \setminus F) : \text{finite } F \subseteq S\bigr\}.
\]
Each member of this family is internal and non-empty.  
For any finite $F_1, F_2$: $\str(S\setminus(F_1\cup F_2)) \subseteq \str(S\setminus F_1) \cap \str(S\setminus F_2)$ and $S\setminus(F_1\cup F_2)$ is non-empty, so
$\mathcal{G}$ has the finite intersection property. 
By Saturation, there exists $y \in \bigcap\mathcal{G}$. 
Then $y \in \str S$, and for every $s \in S$, taking $F = \{s\}$ implies $y \in \str(S \setminus \{s\})$,
so $y \neq s$. Since every standard element of $\str S$ is of the form $s \in S$ (Lemma~\ref{lem:star}(vi)), $y$ must be nonstandard, i.e.\
$y \notin X$.
\end{proof}
\begin{theorem}[$\omega$-accumulation characterisation]
\label{thm:omegaacc}
For every topological model $\mathcal{M}$ and formula $\varphi$:
\[
  x \models \lozengens\varphi
  \iff
 x\in \omega(\sem{\varphi}) 
\]

\end{theorem}

\begin{proof}
Set $A := \sem{\varphi}$.

$(\Rightarrow)$ Suppose $\mu(x) \cap (\str A \setminus X) \neq \emptyset$.
Consider an arbitrary open $U \in \mathcal{N}(x)$.
Then $\str U \supseteq \mu(x)$, so $\str(U \cap A) \setminus X \neq \emptyset$. 
By Proposition~\ref{keyprop}, $U \cap A$ is infinite.

$(\Leftarrow)$ Suppose every $U \in \mathcal{N}(x)$ satisfies $|U \cap A|\geq \aleph_0$.
Consider the following family indexed by all pairs $(U, G)$ with $U \in \mathcal{N}(x)$
and finite $G \subseteq X$:
\[
  \mathcal{F}
  = \bigl\{{}^*(U \cap A) \setminus G :
     U \in \mathcal{N}(x),\, \text{finite } G \subseteq X\bigr\}.
\]
Each member is internal and non-empty. 
Moreover, the finite intersection property holds: given finitely many pairs
$(U_1, G_1), \ldots, (U_n, G_n)$, set $U = U_1 \cap \cdots \cap U_n
\in \mathcal{N}(x)$ and $G = G_1 \cup \cdots \cup G_n$; then
${}^*(U \cap A) \setminus G \subseteq {}^*(U_i \cap A) \setminus G_i$
for each $i$, and this set is non-empty.
By Saturation there exists $y \in \bigcap \mathcal{F}$.

Let $G = \emptyset$; then $y \in \str(U \cap A) \subseteq \str U$ for each $U \in \mathcal{N}(x)$, so $y \in \mu(x)$ and also $y \in \str A$.
Take an arbitrary $a \in X$ and set $G = \{a\}$; then $y \neq a$.
Since $a \in X$ was arbitrary, $y \notin X$.
Thus $y \in \mu(x) \cap (\str A \setminus X)$.
\end{proof}

\subsection{Some properties of the nonstandard semantics}
We show some properties of the nonstandard topological semantics.
We begin by completing the verification that $\lozengens$ defines a normal modal logic.
\begin{lemma}
\label{lem:nsK}
Axiom $K$ is valid for $\lozengens$ on all topological spaces.
\end{lemma}
\begin{proof}
This follows from Proposition~\ref{prop:nsbasic}(ii) and (iii).
\end{proof}
We show that $\lozengens$ coincides with $\lozenge_d$ on $T_1$ spaces; 
the proof illustrates how the halo formulation reduces to the classical characterisation.
\begin{lemma}
\label{thm:coincide}
If $(X,\tau)$ is a $T_1$ topological space then for every model $\mathcal{M}$
and formula $\varphi$:
\[
  \sem{\lozengens\varphi} = \sem{\diam_d\varphi}.
\]
\end{lemma}

\begin{proof}
Since $\omega(A) \subseteq d(A)$, it
suffices to show $d(A) \subseteq \omega(A)$ on
$T_1$ spaces.
Suppose $x \in d(A)$ and $U \in \mathcal{N}(x)$. Suppose for
contradiction that $U \cap A$ is finite: $U \cap A = \{a_1, \ldots, a_k\}
\cup (\{x\} \cap A)$ with $a_1, \ldots, a_k \in A \setminus \{x\}$.
Since $(X,\tau)$ is $T_1$, for each $i$ there exists $W_i \ni x$ open with
$a_i \notin W_i$. Let $W = U \cap W_1 \cap \cdots \cap W_k$. Then
$W \in \mathcal{N}(x)$ and $W \cap (A \setminus \{x\}) = \emptyset$,
contradicting $x \in d(A)$. Hence $U \cap A$ is infinite.

Moreover, Proposition~\ref{prop:nschar}(v) implies $\mu(x) \cap X = \{x\}$
on $T_1$ spaces, so for $y \in \mu(x)$: $y \notin X \iff y \neq x$.
Therefore $\mu(x) \cap (\str A \setminus X) = \mu(x) \cap
(\str A \setminus \{x\})$, which is exactly the equality of the
$\lozengens$ and $\diam_d$ witness conditions.
\end{proof}

\begin{corollary}
\label{cor:T1axioms}
On $T_1$ spaces, $\lozengens$ satisfies Axiom~4.
\end{corollary}

\begin{proof}
On $T_1$ spaces $\lozengens = \diam_d$ by Lemma~\ref{thm:coincide};
since $T_1$ implies $T_D$,  Axiom~4 holds by 
Theorem~\ref{thm:soundd}\mbox{(ii).\qedhere}
\end{proof}
While our study is centred around topological semantics for arbitrary topological spaces, it is worth noting the relationship of $\lozengens$ to both finite topological spaces and Kripke frames.
\begin{proposition}
\label{thm:degenerate}
Given a finite topological space $(X,\tau)$,
$\sem{\lozengens\varphi} = \emptyset$ for every model $\mathcal{M}$ and formula $\varphi$.
\end{proposition}

\begin{proof}
Since each $A\subseteq X$ is finite, by Transfer we obtain $|\str A|=|A|$, so $\str A = A \subseteq X$ and $\str A \setminus X=\emptyset$.
\end{proof}
\begin{corollary}
\label{cor:nofmp}
The logic with $\diam_\mathrm{ns}$ does not have the finite model property.
\end{corollary}

\begin{proposition}
\label{prop:nonKripke}
The operator $\diam_\mathrm{ns}$ cannot be represented as a Kripke modality on any topological space where $\omega$ is non-trivial.
That is, if there exist $x \in X$ and $A \subseteq X$ with
$x \in \omega(A)$, then there is no binary relation
$R \subseteq X \times X$ satisfying
\[
  y \in \omega(B) \iff R(y) \cap B \neq \emptyset
  \quad\text{for all }B \subseteq X,\, y \in X.
\]
\end{proposition}

\begin{proof}
Suppose such $R$ exists. Take any $y_0 \in R(x) \cap A$ (which exists since $x \in \omega(A)$ and the Kripke clause implies $R(x)\cap A \neq\emptyset$). 
Since $\{y_0\}$ is a finite set, $\omega(\{y_0\}) = \emptyset$ by Proposition~\ref{keyprop}: no point can have every neighbourhood meet a singleton in infinitely many points. 
The Kripke clause then implies
$R(x) \cap \{y_0\} = \emptyset$, contradicting $y_0 \in R(x)$.
\end{proof}
Note that Proposition~\ref{prop:nonKripke} in particular holds for every infinite non-discrete $T_1$ space:
Lemma~\ref{thm:coincide} guarantees
$\omega(X\setminus\{x\}) = \mathrm{d}(X\setminus\{x\})
\ni x$ for any non-isolated $x$.

\begin{proposition}
\label{prop:T1char}
The bimodal property
  $\diam_c\varphi \;\leftrightarrow\; \varphi \vee \lozengens\varphi$
is valid in $(X,\tau)$ iff $(X,\tau)$ is $T_1$.
\end{proposition}

\begin{proof}
$(\Leftarrow)$
Let $(X,\tau)$ be $T_1$ and $A = \sem{\varphi}$.
Since $\cl(A) = A \cup d(A)$, by Proposition~\ref{prop:clequiv} we have 
$\sem{\diam_c\varphi} = A \cup d(A)$.
By Lemma~\ref{thm:coincide}, $\omega(A) = d(A)$ on $T_1$ spaces.
Hence $\sem{\varphi \vee \lozengens\varphi}
  = A \cup \omega(A)
  = A \cup d(A)
  = \sem{\diam_c\varphi}$.

$(\Rightarrow)$
Suppose $(X,\tau)$ is not $T_1$. Then some singleton $\{a\}$ is not closed,
so there exists $b \neq a$ with $b \in \cl(\{a\})$.
Set $\nu(p) = \{a\}$, so $A = \{a\}$.
Then $b \in \cl(\{a\}) = \sem{\diam_c p}$.
However, $A = \{a\}$ is finite, so $\str\{a\} \setminus X = \emptyset$
by Proposition~\ref{keyprop}, giving $\omega(\{a\}) = \emptyset$.
Hence $\sem{p \vee \lozengens p} = \{a\} \cup \emptyset = \{a\}$
so $b \notin \sem{p \vee \lozengens p}$.
\end{proof}
The property
$\diam_c\varphi \leftrightarrow \varphi \vee \diam_d\varphi$ in contrast is valid on all topological spaces, since $\cl(A)=A\cup d(A)$.

We now introduce the appropriate notions of perfect sets and scatteredness required for L\"ob Axiom and the decomposition of the $\omega$-accumulation operator.
\begin{definition}[$\omega$-perfect set and $\omega$-scattered space]
\label{def:omegaperf}
A closed subset $P \subseteq X$ is \emph{$\omega$-perfect} if $\omega(P) \supseteq P$,
that is, every point of $P$ is an $\omega$-accumulation point of $P$ itself.
A topological space $(X,\tau)$ is \emph{$\omega$-scattered} if it contains no non-empty $\omega$-perfect subset.
\end{definition}

On $T_1$ spaces, Lemma~\ref{thm:coincide} implies $\omega(A) = d(A)$ for every $A \subseteq X$, so $\omega$-perfect coincides with perfect and $\omega$-scattered coincides with scattered in the sense of Cantor--Bendixson~\cite{Kechris1995}.

\begin{lemma}
\label{lem:omegaclosed}
For every topological space $(X,\tau)$ and every $A \subseteq X$, the set $\omega(A)$ is closed.
\end{lemma}

\begin{proof}
Let $x \in \cl(\omega(A))$ and $U \in \Nbd(x)$.
Pick $y \in U \cap \omega(A)$.
Since $U$ is open and $y \in \omega(A)$, we have $|U \cap A| \geq \aleph_0$.
Since $U$ was arbitrary, $x \in \omega(A)$.
\end{proof}
The Cantor derivative $d(A)$ is not closed in general; closedness requires $T_D$.
In contrast, closedness of $\omega(A)$ holds without any separation assumption (Lemma~\ref{lem:omegaclosed}), yielding the following decomposition, analogous to the classical Cantor--Bendixson theorem~\cite{Kechris1995}.

\begin{theorem}[$\omega$-Cantor-Bendixson]
\label{thm:omegaCB}
Every topological space $(X,\tau)$ has a largest $\omega$-perfect subset $P_\omega$, which is closed.
The complement $X \setminus P_\omega$ is open and $\omega$-scattered.
In particular, $(X,\tau)$ is $\omega$-scattered if and only if $P_\omega = \emptyset$.
\end{theorem}
\begin{proof}
First, note that $\omega$ is monotone: if $A \subseteq B$ then $U \cap A \subseteq U \cap B$ for every open $U$, so $\omega(A) \subseteq \omega(B)$.

Define $\omega^0 = X$, $\omega^{\alpha+1} = \omega(\omega^\alpha)$, and $\omega^\lambda = \bigcap_{\alpha < \lambda} \omega^\alpha$ for limit $\lambda$.
Since $\omega(A)$ is closed (Lemma~\ref{lem:omegaclosed}) and $\omega(A) \subseteq \cl(A)$, if $A$ is closed then $\omega(A) \subseteq A$.
In particular $\omega^1 = \omega(X) \subseteq X = \omega^0$; if $\omega^\alpha \subseteq \omega^\beta$ for $\beta < \alpha$, then by monotonicity $\omega^{\alpha+1} = \omega(\omega^\alpha) \subseteq \omega(\omega^\beta) = \omega^{\beta+1}$, and at limits the inclusion is preserved by definition.
Hence the sequence is decreasing.
Since the $\omega^\alpha$ are subsets of $X$ and each successor step either strictly shrinks the set or stabilises, the sequence must stabilise at some ordinal $\gamma \leq |X|$; set $P_\omega = \omega^\gamma$.
By construction $\omega(P_\omega) = \omega^{\gamma+1} = \omega^\gamma = P_\omega$, so $P_\omega$ is $\omega$-perfect.
Since each $\omega^\alpha$ is closed by Lemma~\ref{lem:omegaclosed}, and intersections of closed sets are closed, $P_\omega$ is closed, and $X \setminus P_\omega$ is open.

It remains to show that $P_\omega$ is the largest $\omega$-perfect subset.
Let $Q \subseteq X$ be $\omega$-perfect, so $Q \subseteq \omega(Q)$.
We show $Q \subseteq \omega^\alpha$ for every $\alpha$ by transfinite induction.
The base case $Q \subseteq \omega^0 = X$ is immediate.
If $Q \subseteq \omega^\alpha$, then by monotonicity $\omega(Q) \subseteq \omega(\omega^\alpha) = \omega^{\alpha+1}$, and since $Q \subseteq \omega(Q)$ we obtain $Q \subseteq \omega^{\alpha+1}$.
At limit $\lambda$, if $Q \subseteq \omega^\alpha$ for all $\alpha < \lambda$, then $Q \subseteq \bigcap_{\alpha < \lambda} \omega^\alpha = \omega^\lambda$.
Hence $Q \subseteq P_\omega$.

Finally, if $P$ is a non-empty $\omega$-perfect subset of the subspace $X \setminus P_\omega$, then $P \subseteq \omega(P)$ since $X \setminus P_\omega$ is open.
The set $\omega(P)$ is then a non-empty closed $\omega$-perfect subset of $X$, so $\omega(P) \subseteq P_\omega$ by maximality, contradicting $P \subseteq X \setminus P_\omega$.
Hence $X \setminus P_\omega$ is $\omega$-scattered.
\end{proof}

\subsection{Topological completeness for K4 and GL}
In this section we establish completeness with respect to the nonstandard semantics.
Corollary~\ref{cor:T1axioms} demonstrates the validity of Axiom~4 for $\lozengens$ on $T_1$ spaces via the collapse to $\diam_d$.
We now show that Axiom~4 holds on all topological spaces, with no separation condition required.
\begin{lemma}
\label{thm:nsK4}
For every topological model $\mathcal{M}$ and every
formula $\varphi$:
\[
  \sem{\lozengens\lozengens\varphi}
  \subseteq
  \sem{\lozengens\varphi}.
\]
\end{lemma}

\begin{proof}
Let $A = \sem{\varphi}$ and suppose $x \in \omega
(\omega(A))= \sem{\lozengens\lozengens\varphi}$ by Theorem~\ref{thm:omegaacc}. 
Let $U \in \mathcal{N}(x)$ be arbitrary.
By definition, $|U\cap \omega(A)|\geq\aleph_0$; pick a point $y \in U \cap \omega(A)$. 
Since $U$ is an open neighbourhood of $y$ and $y \in \omega(A)$,
$|U \cap A|\geq \aleph_0$. 
Since $U$ is arbitrary, $x \in
\omega(A) = \sem{\lozengens\varphi}$ by Theorem~\ref{thm:omegaacc}.
\end{proof}
\begin{theorem}[topological completeness for K4]
\label{thm:nscompleteness}
$\mathsf{K4}$ is sound for $\lozengens$ over all topological spaces. 
It is complete over all infinite topological spaces.
\end{theorem}
\begin{proof}
Soundness follows from the results above.
For completeness, suppose $\varphi \notin \mathsf{K4}$.
By the Kripke completeness of $\mathsf{K4}$~\cite{BlackburnDeRijkeVenema2001}, $\varphi$ is 
falsified on some finite transitive Kripke frame.
By Bezhanishvili, Esakia, and Gabelaia~\cite{BEG2005,BEG2010}, every such frame 
is the image of a d-map from a metrizable, hence infinite and $T_1$, 
space $(X,\tau)$, and $\varphi$ is falsified under $\diam_d$ on $(X,\tau)$.
Since $(X,\tau)$ is $T_1$, Lemma~\ref{thm:coincide} implies $\lozengens = \diam_d$, so $\varphi$ is falsified under the nonstandard semantics on an infinite topological space.
\end{proof}
We proceed to the next completeness result. 
\begin{lemma}
\label{thm:GLchar}
The \textup{L}  axiom
\[
  \sq(\sq\varphi \to \varphi)
  \to
  \sq\varphi
\]
is valid for the nonstandard semantics over all $\omega$-scattered topological spaces.
\end{lemma}

\begin{proof}
Fix a model $\mathcal{M}$ and set $A = \sem{\neg\varphi}^{\mathcal{M}}$.
By the $\omega$-accumulation characterisation,  
$\sem{\sq_\mathrm{ns}\varphi}^{\mathcal{M}} = X \setminus \omega(A)$.
From this we obtain $ \sem{\sq_\mathrm{ns}(\sq_\mathrm{ns}\varphi \to \varphi)}
  = X \setminus \omega(A \setminus \omega(A))$.
Hence L is false at $x \in X$ if and only if $x$ belongs to the set $D := \omega(A) \setminus \omega(A \setminus \omega(A))$.
In particular, L is valid in $\mathcal{M}$ if and only if $D = \emptyset$.

\medskip
$(\Leftarrow)$
Suppose $(X,\tau)$ is $\omega$-scattered. We show $D = \emptyset$ by proving that $D$ is $\omega$-perfect; $\omega$-scatteredness then forces $D = \emptyset$.
Let $x \in D$. 
By definition, since $x \in \omega(A)$, every open $U\ni x$ satisfies
    $|U \cap A| \geq \aleph_0$; since $x \notin \omega(A \setminus \omega(A))$, there is an open
    neighbourhood $U_x$ of $x$ with
    $\bigl|U_x \cap (A \setminus \omega(A))\bigr| < \aleph_0$.
Let $V$ be an arbitrary open neighbourhood of $x$; replacing it with $V \cap U_x$ if necessary, we may assume $V \subseteq U_x$.
Since $x \in \omega(A)$, the set $V \cap A$ is infinite.
Since $V \subseteq U_x$, the set $V \cap (A \setminus \omega(A))$ is finite.
Therefore
\[
  V \cap A \cap \omega(A)
  = (V \cap A) \setminus (A \setminus \omega(A))
\]
is infinite.
We claim that every $y \in V \cap A \cap \omega(A)$ belongs to $D$, i.e.\
satisfies $y \notin \omega(A \setminus \omega(A))$.
Suppose for contradiction that $y \in \omega(A \setminus \omega(A))$.
Since $y \in V$ and $V$ is open, $\bigl|V \cap (A \setminus \omega(A))\bigr| \geq \aleph_0$.
But $V \subseteq U_x$ implies
$V \cap (A \setminus \omega(A)) \subseteq U_x \cap (A \setminus \omega(A))$,
which is finite, a contradiction.

Therefore $V \cap D \supseteq V \cap A \cap \omega(A)$, which is infinite.
Since $V$ was an arbitrary open neighbourhood of $x$, we have $x \in \omega(D)$.
As $x \in D$ was arbitrary, $D \subseteq \omega(D)$, so $\omega(D)$ is a closed $\omega$-perfect subset of $X$.
Since $(X,\tau)$ is $\omega$-scattered, $\omega(D) = \emptyset$, hence $D = \emptyset$, so L is valid in $\mathcal{M}$.

($\Rightarrow$) 
Suppose $P \subseteq X$ is a non-empty $\omega$-perfect set, so
$\omega(P) \supseteq P$. Define a model $\mathcal{M}$ such that $\nu(p) = X \setminus P$,
and set $A := \sem{\neg p} = P$.
Since $\omega(P) \supseteq P$,
\[
  \sem{\sq_\mathrm{ns} p} = X \setminus \omega(P) \subseteq X \setminus P = \sem{p},
\]
so $\sq_\mathrm{ns} p \to p$ is valid on $\mathcal{M}$.
Hence $\sem{\neg(\sq_\mathrm{ns} p \to p)} = \emptyset$ and so 
$\omega(\emptyset) = \emptyset$ implies
$\sem{\sq_\mathrm{ns}(\sq_\mathrm{ns} p \to p)} = X \setminus \omega(\emptyset) = X$.
On the other hand, for each $x \in P \subseteq \omega(P)=\omega(A)$, we have $x \notin X \setminus \omega(A) = \sem{\sq_\mathrm{ns} p}$.
Therefore, L is not valid on $(X,\tau)$.
\end{proof}

\begin{theorem}[topological completeness for GL]
\label{cor:GLcomplete}
$\mathsf{GL}$ is sound for $\lozengens$ over all $\omega$-scattered topological spaces. 
It is complete over all infinite $\omega$-scattered topological spaces. 
\end{theorem}

\begin{proof}
Soundness follows from the results above.
For completeness, suppose $\varphi \notin \mathsf{GL}$.
By the Kripke completeness of $\mathsf{GL}$~\cite{segerberg}, $\varphi$ 
is falsified on some finite irreflexive transitive Kripke frame.
By Bezhanishvili, Esakia, and Gabelaia~\cite{BEG2005,BEG2010}, every such frame 
is the image of a d-map from a scattered metrizable, hence infinite and 
$T_1$, space $(X,\tau)$, and $\varphi$ is falsified under $\diam_d$ 
on $(X,\tau)$.

Since $(X,\tau)$ is $T_1$, Lemma~\ref{thm:coincide} gives 
$\lozengens = \diam_d$, so $\varphi$ is falsified under the 
nonstandard semantics.
It remains to show $(X,\tau)$ is $\omega$-scattered.
If $P$ is a non-empty $\omega$-perfect subset, then $P\subseteq \omega(P)=d(P)$ by Lemma~\ref{thm:coincide}, so $P$ is a non-empty perfect set, contradicting scatteredness.
Therefore $(X,\tau)$ is $\omega$-scattered, and $\varphi$ is falsified under $\lozengens$ on an infinite $\omega$-scattered space.
\end{proof}
The argument above shows in particular that every scattered space is $\omega$-scattered, but not vice versa.

\section{Final remarks}

We have introduced the halo operator, giving a uniform nonstandard-analytic framework from which the classical closure and derived-set semantics arise as instances. 
It yields a new operator which turns out to be the modal counterpart of the $\omega$-accumulation point operator, with $\mathsf{K4}$ and $\mathsf{GL}$ as its complete logics over infinite spaces and infinite $\omega$-scattered spaces, respectively.

The four exclusion functions studied here are the canonical extremes, but intermediate choices are a natural next step.
 For instance, $E(x) = \{x\} \cup (\mu(x) \cap X)$ excludes the point itself together with all standard points in its halo.
 More generally, one could require witnesses to avoid any prescribed internal subset of $\mu(x)$, producing a rich family of operators interpolating between $\diam_d$ and $\lozengens$.

Another interesting direction for future work is the logic of $\lozengens$ over a single fixed infinite space such as $\mathbb{R}$.
    The natural candidate is $\mathsf{K4}$ extended by the seriality axiom $\lozengens\top$, since $\mathbb{R}$ is dense and every point is an $\omega$-accumulation point of any infinite set, but a precise
    axiomatisation and proof remain to be established.

\section*{Acknowledgements}

Yo\`av Montacute is supported by ACT-X, Grant No.\ JPMJAX24CR, JST, Japan; and by ASPIRE, Grant No.\ JPMJAP2301, JST, Japan.


\bibliographystyle{eptcs}
\bibliography{generic}
\end{document}